\documentclass[a4paper,10pt, leqno]{amsart}
\usepackage{ascmac}
\usepackage{amssymb}
\usepackage{amsmath}
\usepackage{amsthm}
\usepackage{graphicx}
\newtheorem{thm}{Theorem}[section]
\newtheorem{lem}[thm]{Lemma}

\newtheorem{rmk}{Remark}

\numberwithin{equation}{section}

\newcommand{\la}{\lambda}
\newcommand{\va}{\varphi}

\newcommand{\wa}{\widetilde{a}}
\newcommand{\wb}{\widetilde{b}}
\newcommand{\wc}{\widetilde{c}}
\newcommand{\wmm}{\widetilde{m}}

\newcommand{\wchi}{\widetilde{\chi}}
\newcommand{\wf}{\widetilde{f}}
\newcommand{\wv}{\widetilde{v}}

\DeclareMathOperator{\supp}{supp}

\newcommand{\R}{\mathbb{R}} 
\newcommand{\N}{\mathbb{N}} 
\newcommand{\pp}{\partial}

\newcommand{\relmiddle}[1]{\mathrel{}\middle#1\mathrel{}}

\allowdisplaybreaks

\title[]{%
H\"{o}lder stability estimate in an inverse source problem 
for a first and half order time fractional diffusion equation
}
\author{Atsushi Kawamoto}
\address{ Department of Mathematical Sciences, The University
of Tokyo, Komaba Meguro Tokyo 153 Japan,
E-mail:kawamo@ms.u-tokyo.ac.jp
}
\date{}
\begin{document}

\begin{abstract}
We consider the first and half order time fractional equation with the zero initial condition. 
We investigate an inverse source problem of determining the time-independent source factor 
by the spatial data at an arbitrarily fixed time 
and we establish the conditional stability estimate of H\"{o}lder type in our inverse problem. 
Our method is based on the Bukhgeim-Klibanov method by means of the Carleman estimate. 
We also derive the Carleman estimate for the first and half order time fractional diffusion equation.  

\end{abstract}
\markright{\today}
\maketitle
%
%
\noindent

\section{Introduction}


Let 
$T>0$, 
$\Omega\subset \R^n$ be a bounded domain with sufficiently smooth boundary $\pp\Omega$. 
We set $Q=\Omega \times (0,T)$. 

We consider 
the following equation and the initial condition:
\begin{align}
\label{eq:L2_eq}
&
(\rho_1 \pp_t+\rho_2\pp_t^{\frac12}  - L)u(x,t)=g(x,t),\quad
(x,t)\in Q,\\
\label{eq:ini_condi}
&
u(x,0)=0,\quad x\in \Omega,
\end{align}
where 
$\rho_1>0$, $\rho_2 \neq 0$ are constants, 
$\pp_t^{\frac12}$ is a Caputo type fractional derivative of half order:
\begin{equation*}
\pp_t^{\frac12} u(x,t) := \frac{1}{\Gamma\left(\frac12\right)} \int_0^t \frac{\pp_t u(x,\tau)}{(t-\tau )^{\frac12}}\,d\tau,
\quad (x,t)\in Q,
\end{equation*}
and $L$ is a symmetric uniformly elliptic operator:
\begin{equation}
\label{eq:def_L2}
L u(x,t) :=\sum_{i,j=1}^n \pp_i (a_{ij}(x) \pp_j u(x,t))
+\sum_{j=1}^n b_j (x)\pp_j u(x,t)
+ c(x)u(x,t),\ (x,t) \in Q.  
\end{equation}
Throughout this article, 
we assume that $a_{ij}\in C^3(\overline{\Omega})$, $a_{ij}=a_{ji}$ ($1\leq i,j \leq n$), 
and that there exists a constant $m>0$ such that 
\begin{equation}
\label{eq:coe_a}
\frac1{m} |\xi|^2 
\leq
\sum_{i,j=1}^n a_{ij}(x) \xi_i \xi_j
\leq
m |\xi|^2,
\quad
\xi=(\xi_1,\ldots, \xi_n) \in \R^n,\ 
x \in \overline{\Omega},
\end{equation}
and 
$
b_j \in  C^2(\overline{\Omega})$ ($1\leq j \leq n$), $c \in  C^2(\overline{\Omega}).
$
Here and Henceforth we use notations
$\pp_t= \frac{\pp}{\pp t}$, 
$\pp_i=\frac{\pp}{\pp x_i}$ ($i=1,2,\ldots, n$). 
We also use the multi index $\alpha=(\alpha_1,\alpha_2, \ldots, \alpha_n)$ 
with $\alpha_j \in \mathbb{N}\cup\{0\}$ ($j=1,2,\ldots,n$), 
$\pp_x^\alpha =\pp_1^{\alpha_1}\pp_2^{\alpha_2} \cdots \pp_n^{\alpha_n}$, 
$|\alpha|=\alpha_1+\alpha_2+\cdots+\alpha_n$. 
Let $\nu =\nu (x)$ be the outwards unit normal vector to $\pp\Omega$ at $x$ and let 
$\pp_\nu =\nu \cdot \nabla$.


Recent years, anomalous diffusion phenomena are actively studied. 
If we consider the diffusion phenomena in some heterogeneous media, 
it is known that the diffusion  is different from the classical diffusion 
and it is called the anomalous diffusion.  
Their researches are important in various applications such as environmental problems. 
Indeed, anomalous diffusion phenomena appear naturally if we study the pollution in the underground. 
As a mathematical approach for the anomalous diffusion phenomena, we may consider a homogenization. 
In this article, we treat the equation \eqref{eq:L2_eq} which is derived from the homogenization by 
Amaziane, Pankratov and Piatnitski \cite{APP}. 
As a micro model, they considered the linear parabolic equation in thin periodic fractured media, 
and then, they established a homogenized  macro model for \eqref{eq:L2_eq}. 
Meanwhile, one of the popular model equations for the anomalous diffusion phenomena
is a fractional diffusion equation (FDE) derived from the continuous time random walk model. 
It may explain field data of the anomalous diffusion which is slower than the classical one (see e.g., \cite{AG, HH, MK}). 
FDEs and their applications for inverse problems are investigated intensively (see e.g., \cite{JR} and references therein). 
As the other interpretation for our equation, we may regard \eqref{eq:L2_eq} as the special case of 
the multi term time FDE which is known as a generalization of FDEs. 
Here we call \eqref{eq:L2_eq} a first and half order time fractional diffusion equation. 
On the well-posedness results for multi term time FDEs, 
we may refer to \cite{BD, LLY, L}. 


In this article we consider an inverse problem of determining the time-independent factor  
of the source term $g$ of \eqref{eq:L2_eq}  
by the additional data $u(x,t_0)$, $x\in \Omega$ where $t_0\in (0,T)$ is an arbitrarily fixed time. 
We establish the conditional H\"{o}lder type stability estimate in our inverse source problem. 


Our methodology is based on the Bukhgeim-Klibanov method. 
Bukhgeim and Klibanov proved the global uniqueness in inverse problems 
by using the Carleman estimate in \cite{BK}. 
Imanuvilov and Yamamoto \cite{ImY} established global Lipschitz type stability estimates 
in inverse source problems for parabolic equations. 
As for the Bukhgeim-Klibanov method, see some monographs and papers 
\cite{K1, K2, KT, Y} and references therein. 


One of the key tools in the Bukhgeim-Klibanov method is the Carleman estimate. 
A Carleman estimate is a weighted $L^2$ inequality for a solution of a partial differential equation 
and is introduced by Carleman \cite{C} to prove the uniqueness in the Cauchy problem for a first order elliptic system in $\R^2$ with non-analytic coefficients.  
Since Carleman's paper, there have been a lot of work for Carleman estimates and their applications
(see e.g., monographs \cite{H, Is1, KT} and references therein). 
In this article we establish the local Carleman estimate for \eqref{eq:L2_eq} with the zero initial condition 
by using the local Carleman estimate for parabolic equations. 
The main idea of the proof is a transformation from \eqref{eq:L2_eq} to an integer order partial differential equation. 
Related to the Carleman esitmate for FDEs, we may refer to \cite{CLNa, Kawa, LNa2, XCY} 
in which the single term time FDE is considered.


As the most relevant work to this paper, we may refer to \cite{Kawa, KM, YZ}. 
Yamamoto and Zhang \cite{YZ} established the conditional H\"{o}lder type stability estimate in determining the time-independent source factor 
of a single term time FDE from the additional data at an arbitrarily fixed time 
by using the local Carleman estimate derived in \cite{XCY}.  
Our result for the stability estimate is motivated by their result. 
In \cite{Kawa}, he obtained the Lipschitz type stability estimates by 
the additional data at an arbitrarily fixed time and 
boundary/interior data 
in inverse source problems for a single term time FDE
with suitable boundary conditions and the zero initial condition 
by using the Carleman estimates with boundary/interior data. 
In \cite{KM}, they investigated the half order time fractional radiative transport equation which is a model equation of anomalous transport phenomena derived in \cite{M}, 
and then, they proved Lipschitz type stability estimate in coefficient inverse problems by using the Carleman estimate with boundary data. 
In the above articles \cite{Kawa, KM, YZ}, they dealt with one dimensional case in space.


To the author's knowledge, stability estimates for multi term time FDEs 
have not established yet, although \eqref{eq:L2_eq} is a special case of multi term time FDEs. 
Moreover our equation \eqref{eq:L2_eq} is multi-dimensional case in space.


On inverse problems for multi term time FDEs, 
we may refer to \cite{JLLY, LImY, LY, Liu}. 
Related to our results for the inverse source problems, 
Liu \cite{Liu} obtained the uniqueness result in inverse problem of determining the temporal component of the source term from the single point observation by using the strong maximum principle. 
In \cite{JLLY}, they established the uniqueness in determining the spatial component of the source term from interior observation.


This article is organized as follows. 
\S 1 is devoted to this introduction. 
In \S 2 we describe the main result for the stability estimate in an inverse source problem.  
We state and show the Carleman estimate in \S3. 
And we prove our main Theorem of the stability estimate by using the Carleman estimate in \S4. 

\section{stability estimate}

Let $t_0\in (0,T)$ be an arbitrarily fixed time. 
Let $\gamma$ be an arbitrarily fixed open connected sub-boundary of $\pp\Omega$
and let $\omega\Subset\Omega \cup \gamma$ be a sub-domain with the smooth boundary $\pp\omega$. 
We set $\Gamma=\gamma \times (0,T)$. 
We assume that 
\begin{equation}
\label{eq:R}
\left\{
\begin{aligned}
&R\in C^2([0,T];C^2(\overline{\Omega}))\cap C^3([0,T];C(\overline{\Omega})),\quad
\pp_t^{\frac12} R \in C^2([0,T];C(\overline{\Omega})),\\
&\text{and}\ |R(x,t_0)|>0,\  x\in\overline{\Omega}.
\end{aligned}
\right.
\end{equation}

We consider 
the following equation and the initial condition:
\begin{align}
\label{eq:eq01}
&
(\rho_1\pp_t +\rho_2\pp_t^{\frac12} - L)u(x,t)=f(x)R(x,t),&
(x,t)\in Q,\\
\label{eq:eq02}
&
u(x,0)=0,& x\in \Omega,\\
&
u(x,t)=\pp_\nu u(x,t)=0,&
(x,t)\in \Gamma, 
\label{eq:eq03}
\end{align}
and we investigate the following inverse problem. 

\noindent
{\bf Inverse source problem}: 
Determine $f(x)$, $x\in \omega$ by data $u(x,t_0)$, $x\in \Omega$. 
\begin{thm}
\label{thm:isp}
Let $u, \pp_t u, \pp_t^2 u \in L^2(0,T;H^4(\Omega))\cap H^1(0,T;H^2(\Omega))\cap  H^2(0,T;L^2(\Omega))$ and u satisfies \eqref{eq:eq01}--\eqref{eq:eq03} and the additional boundary condition $\pp_x^\alpha u(x,t)=0$, $(x,t)\in \Gamma$, $|\alpha|=2$. 
We suppose that $f\in H^2(\Omega)$ with $f=\pp_\nu f=0$ on $\gamma$ 
and $R$ satisfies \eqref{eq:R},  
Moreover we assume that there exists $M>0$ such that
\begin{equation}
\label{eq:M}
\left\{
\begin{aligned}
&
\|\pp_t^k u\|_{L^2(0,T;H^4(\Omega))}
+
\|\pp_t^ku\|_{H^1(0,T;H^2(\Omega))}\\
&\qquad
+
\|\pp_t^ku\|_{H^2(0,T;L^2(\Omega))} 
\leq M,\quad k=0,1,2,\\
& 
\| f \|_{H^2(\Omega)}
\leq M.
\end{aligned}
\right.
\end{equation} 
Then there exist constants $C>0$ and $\kappa \in (0,1)$ such that
\begin{equation}
\label{eq:se}
\| f \|_{H^2(\omega)}
\leq
C
\|u(\cdot, t_0)\|_{H^4(\Omega)}^\kappa.  
\end{equation}
\end{thm}
\begin{rmk}
For simplicity of the statement of Theorem \ref{thm:isp}, we observed the additional data $u(\cdot,t_0)$ on the whole domain $\Omega$. 
More precisely,  it is sufficient to observe the data $u(\cdot,t_0)$ on the sub-domain $\Omega_2$ of $\Omega$ such that $\omega \subset \Omega_2$. 
$\Omega_2$ is defined in the proof of this Theorem and 
the domain depends on the weight function of the Carleman estimate. 
\end{rmk}
\begin{rmk}
It is required to assume that the additional boundary condition $\pp_x^\alpha u(x,t)=0$, $(x,t)\in \Gamma$, $|\alpha|=2$ 
when we estimate the boundary term of the Carleman estimate in the proof of our Theorem \ref{thm:isp}. 
In the case of $n=1$, that is, $\Omega\subset \R$ is one dimensional case in space, 
we may relax this additional boundary condition. 
Indeed, by $f=0$ on $\gamma$ and $u=\pp_1 u=0$ on $\Gamma$, 
\eqref{eq:eq01} implies that $\pp_1^2 u=0$ on $\Gamma$. 
In the multi dimensional case in space, taking a small sub-domain $\widetilde{\omega} \subset \Omega$ such that $\Gamma\subset\pp\widetilde{\omega}$, 
we may also assume that $u(x,t)=0$, $(x,t)\in \widetilde{\omega}$ instead of boundary conditions. 
\end{rmk}
\begin{rmk}
Comparing our result (Theorem \ref{thm:isp}) with the result by Yamamoto and Zhang \cite{YZ}, 
they proved the conditional stability estimate of H\"{o}lder type in the case of $\rho_1=0$, $\rho_2=1$ in \eqref{eq:L2_eq} and $n=1$. 
On the other hand, 
we assume that $\rho_1>0$, $\rho_2\neq 0$ in \eqref{eq:L2_eq} and $n \in \N$ in our result.  
\end{rmk}

\section{Carleman estiamte}

We reduce \eqref{eq:L2_eq} to the following equation:
\begin{equation*}
\rho_2^2 \pp_t u(x,t)- (\rho_1\pp_t - L)^2 u(x,t)=G(x,t),\quad (x,t)\in Q.
\end{equation*}
And we derive the Carleman estimate for the above equation. 

Let us justify the above reduction from \eqref{eq:L2_eq}. 
Using the idea of Xu, Cheng and Yamamoto \cite{XCY}, we may prove the following Lemma. 
\begin{lem}
\label{lem:halftoone}
If $u \in C([0,T];H^4(\Omega))\cap  C^1([0,T];H^2(\Omega))\cap  C^2([0,T];L^2(\Omega))$ satisfies \eqref{eq:L2_eq} and \eqref{eq:ini_condi}, 
then $u$ satisfies
\begin{equation}
\label{eq:lem01}
\rho_2^2 \pp_t u(x,t)- (\rho_1\pp_t - L)^2 u(x,t)=G(x,t),\quad (x,t)\in Q
\end{equation}
where 
\begin{equation}
\label{eq:lem02}
G(x,t)=\left[ \rho_2\pp_t^\frac12- (\rho_1\pp_t - L) \right]g(x,t) +\frac{\rho_2 g(x,0)}{\sqrt{\pi t}},\quad (x,t)\in Q. 
\end{equation}
\end{lem}
\begin{proof}[Proof of Lemma]
Applying $\left[ \rho_2\pp_t^\frac12- (\rho_1\pp_t - L) \right]$ to \eqref{eq:L2_eq}, we have
\begin{equation*}
\left[ \rho_2\pp_t^\frac12- (\rho_1\pp_t - L) \right](\rho_2\pp_t^{\frac12}+\rho_1\pp_t  - L)u
=\left[ \rho_2\pp_t^\frac12- (\rho_1\pp_t - L) \right] g.
\end{equation*}
Expanding the left-hand side of the above equation, we get
\begin{equation*}
\rho_2^2\pp_t^\frac12\pp_t^\frac12 u 
+ \rho_1\rho_2\pp_t^\frac12 \pp_t u 
- \rho_1\rho_2\pp_t \pp_t^\frac12 u
-\rho_1^2\pp_t^2 u +2 \rho_1\pp_t L u - L^2 u
=\left[ \rho_2\pp_t^\frac12- (\rho_1\pp_t - L) \right] g. 
\end{equation*}
Since $(\rho_1\pp_t -L)^2=\rho_1^2\pp_t^2 -2\rho_1\pp_t L +L^2$, we have
\begin{equation}
\label{eq:l00}
\rho_2^2\pp_t^\frac12\pp_t^\frac12 u 
+ \rho_1\rho_2\pp_t^\frac12 \pp_t u 
- \rho_1\rho_2\pp_t \pp_t^\frac12 u
-(\rho_1\pp_t -L)^2 u
=\left[ \rho_2\pp_t^\frac12- (\rho_1\pp_t - L) \right] g. 
\end{equation}

Henceforth we calculate $\pp_t^\frac12\pp_t^\frac12 u$,  $\pp_t^\frac12 \pp_t u - \pp_t \pp_t^\frac12 u$. 

Let $D_t^\alpha$ denote the Riemann-Liouville fractional derivative. 
By the relations between the Riemann-Liouville and the Caputo fractional derivative (see \cite{Pod}), we have
\begin{equation}
\label{eq:l01}
\pp_t^{\frac12} u (x,t) =D_t^\frac12 u(x,t) - \frac{u(x,0)}{\sqrt{\pi t}},\quad (x,t)\in Q
\end{equation} 
and
\begin{equation}
\label{eq:l02}
D_t^\frac12 D_t^\frac12 u (x,t)= \pp_t u(x,t),\quad (x,t)\in Q.
\end{equation}

By \eqref{eq:l01}  and \eqref{eq:ini_condi}, we obtain 
\begin{equation}
\label{eq:l03}
\pp_t^{\frac12} u (x,t) =D_t^\frac12 u(x,t),\quad (x,t)\in Q. 
\end{equation}
Setting $v=\pp_t^\frac12 u$, \eqref{eq:l01} gives us
\begin{equation}
\label{eq:l04}
\pp_t^{\frac12} v (x,t) =D_t^\frac12 v(x,t) - \frac{v(x,0)}{\sqrt{\pi t}},\quad (x,t)\in Q.
\end{equation}
By \eqref{eq:l03} and \eqref{eq:l02}, we see that
\begin{equation*}
D_t^\frac12 v(x,t)=D_t^\frac12 \pp_t^\frac12 u(x,t)=D_t^\frac12 D_t^\frac12 u(x,t)=\pp_t u(x,t),\quad (x,t)\in Q.
\end{equation*}
Together this with \eqref{eq:l04} , we have
\begin{equation}
\label{eq:l06}
\pp_t^\frac12 \pp_t^\frac12 u(x,t) = \pp_t  u(x,t)- \frac{\pp_t^\frac12u (x,0)}{\sqrt{\pi t}},\quad (x,t)\in Q.
\end{equation}

By the definition of the Caputo derivative,
\begin{equation}
\label{eq:l07}
\pp_t^\frac12 \pp_t u(x,t) 
=\frac{1}{\Gamma \left( \frac12 \right)} \int_0^t \frac{\pp_t^2 u(x,\tau)}{(t-\tau)^\frac12} \,d\tau,\quad (x,t)\in Q
\end{equation}
and 
\begin{align*}
\pp_t^\frac12 u(x,t)
&=\frac{1}{\Gamma \left( \frac12 \right)} \int_0^t \frac{\pp_t u(x,\tau)}{(t-\tau)^\frac12} \,d\tau\\
&=\frac{1}{\Gamma \left( \frac12 \right)} \left( 2t^\frac12 \pp_t u (x,0)+2 \int_0^t (t-\tau)^\frac12 \pp_t^2u(x,\tau) \,d\tau \right),\quad (x,t)\in Q.
\end{align*}
Therefore
\begin{equation}
\label{eq:l08}
\pp_t\pp_t^\frac12 u(x,t)
=\frac1{\Gamma \left( \frac12 \right)} \left( \frac1{t^\frac12} \pp_t u (x,0)+\int_0^t \frac{\pp_t^2 u(x,\tau)}{(t-\tau)^\frac12} \,d\tau \right),\quad (x,t)\in Q.
\end{equation}
By \eqref{eq:l07} and \eqref{eq:l08}, we obtain
\begin{equation}
\label{eq:l09}
\pp_t^\frac12 \pp_t u(x,t)  - \pp_t\pp_t^\frac12 u(x,t)
=-\frac{\pp_t u (x,0)}{\sqrt{\pi t}},\quad (x,t)\in Q.
\end{equation}
Here we note that $\Gamma \left( \frac12 \right) =\sqrt{\pi}$. 

\eqref{eq:l06} and \eqref{eq:l09} yield
\begin{align}
\label{eq:l10}
&\rho_2^2\pp_t^\frac12\pp_t^\frac12 u(x,t)
+ \rho_1\rho_2\pp_t^\frac12 \pp_t u(x,t) 
- \rho_1\rho_2\pp_t\pp_t^\frac12 u(x,t)
\\
&=
\rho_2^2\pp_t u(x,t)
-\frac{\rho_2^2\pp_t^\frac12 u(x,0)}{\sqrt{\pi t}}
-\frac{\rho_1\rho_2\pp_t u (x,0)}{\sqrt{\pi t}}
\nonumber \\
&=\rho_2^2\pp_t u(x,t)
-\frac{\rho_2}{\sqrt{\pi t}} \left( \rho_2\pp_t^\frac12 u(x,0) +\rho_1\pp_t u (x,0) \right),\quad (x,t)\in Q. 
\nonumber 
\end{align}
By  \eqref{eq:L2_eq} and \eqref{eq:ini_condi}, we see that
\begin{equation*}
\rho_2\pp_t^\frac12 u(x,0)+\rho_1\pp_t u(x,0)=g(x,0),\quad x\in \Omega.
\end{equation*}
Combining this with \eqref{eq:l10}, we have
\begin{align}
\label{eq:l12}
&\rho_2^2\pp_t^\frac12\pp_t^\frac12 u(x,t)
+ \rho_1\rho_2\pp_t^\frac12 \pp_t u(x,t) 
- \rho_1\rho_2\pp_t\pp_t^\frac12 u(x,t) \\
&= \rho_2^2\pp_t u(x,t)-\frac{\rho_2 g(x,0)}{\sqrt{\pi t}},\quad (x,t)\in Q. \nonumber
\end{align}

By \eqref{eq:l00} and \eqref{eq:l12}, we conclude that
\begin{equation*}
\rho_2^2 \pp_t u(x,t) 
-(\rho_1\pp_t -L)^2 u(x,t)
=\left[ \rho_2\pp_t^\frac12- (\rho_1\pp_t - L) \right]g(x,t) +\frac{\rho_2 g(x,0)}{\sqrt{\pi t}},\quad (x,t)\in Q. 
\end{equation*}. 
Thus we complete the proof. 
\end{proof}

Let us introduce the weight function for the Carleman estimate. 
Let $\beta > 0$, $t_0\in (0,T)$. 
Let $d\in C^2(\overline{\Omega})$ satisfy $|\nabla d|\neq 0$ in $\overline{\Omega}$. 
We set the weight function:
\begin{equation*}
\varphi(x,t)=e^{\la \psi(x,t)}, \quad \psi(x,t)=d(x)-\beta(t-t_0)^2, \quad (x,t) \in Q .
\end{equation*}
Let $V\subset Q$ be a sub-domain with the smooth boundary $\pp V$. 
And we set $w(x,t)=\rho_1\pp_t u(x,t)-Lu(x,t)$, $(x,t)\in Q$. 

Now we ready to state the Carleman estimate for \eqref{eq:L2_eq}. 

\begin{thm}
\label{thm:ce0}
There exists $\la_0>0$ such that for any $\la>\la_0$, 
we can choose $s_0(\la)>0$ satisfying:
there exists $C=C(s_0,\la_0)>0$ such that 
\begin{align}
\label{eq:pce}
&
\int_V
\Biggl[
\frac1{s^2}  
\left(
|\pp_t^2 u|^2 
+
\sum_{i,j=1}^n|\pp_t\pp_i \pp_j u|^2
\right)
+\la^2|\nabla \pp_t u|^2  \\
&\qquad
+s^2  \la^4\left( |\pp_t u|^2 
+\sum_{i,j=1}^n|\pp_i \pp_j u|^2
\right)
+s^4 \la^6  |\nabla u|^2 
+s^6 \la^8 |u|^2
\Biggr]
e^{2s\va}\,dxdt 
\nonumber\\
&
\leq 
C
\int_V \left|\left[\rho_2^2\pp_t - (\rho_1\pp_t - L)^2\right]u\right|^2 e^{2s\va}\,dxdt
+Ce^{Cs} B
\nonumber 
\end{align}
for 
all $s> s_0$ and 
all $u \in L^2(0,T;H^4(\Omega))\cap H^1(0,T;H^2(\Omega))\cap  H^2(0,T;L^2(\Omega))$, 
where $B$ is the following boundary term:
\begin{equation*}
B=
\int_{\pp V}
\left(
|\nabla w|^2 + |\pp_t w|^2 + |w|^2 
+|\nabla \pp_t u|^2 + |\nabla u|^2 
+|\pp_t^2 u|^2 +|\pp_t u|^2 + |u|^2
\right)
\,dsdt .
\end{equation*}
\end{thm}
%
%

As Lemma \ref{lem:celem} and Lemma \ref{lem:celeme}, 
we introduce previous results of Carleman estimates with two large parameters 
for the second order partial differential equations. 
As for the proofs of two lemmas, see e.g., \cite{EI, IK1, IK2, Y}.  
Eller and Isakov \cite{EI} established Carleman estimates by using differential quadratic forms, 
an approach of H\"{o}rmander \cite{H}. 
On the other hand, we see the direct derivation of Carleman estimates by using integrating by parts in Yamamoto \cite{Y}.  

To prove the above Theorem \ref{thm:ce0},
we use the following Carleman estimate for parabolic equations. 
\begin{lem}
\label{lem:celem}
There exists $\la_0>0$ such that for any $\la>\la_0$, 
we can choose $s_0(\la)>0$ satisfying: 
there exists $C=C(s_0,\la_0)>0$ such that 
\begin{align*}
&
\int_V
\left[
\frac1{s}  
\left(
|\pp_t v|^2 
+
\sum_{i,j=1}^n|\pp_i \pp_j v|^2
\right)
+s\la^2 |\nabla v|^2 
+s^3 \la^4 |v|^2
\right]
e^{2s\va}\,dxdt \\
&
\leq 
C\int_V |(\rho_1\pp_t -L) v|^2 e^{2s\va}\,dxdt
+
Ce^{Cs}
\int_{\pp V}
\left(
|\nabla v|^2 + |\pp_t v|^2 + |v|^2 
\right)
\,dsdt
\end{align*}
for 
all $s> s_0$ and 
all $v \in L^2(0,T;H^2(\Omega))\cap  H^1(0,T;L^2(\Omega))$. 
\end{lem}

Moreover we state the Carleman estimate for the elliptic equation 
which we use in the proof of the stability estimate in our inverse problem. 

Let $\widetilde{L}$ be a symmetric uniformly elliptic operator:
\begin{equation*}
\widetilde{L} \wv(x) :=\sum_{i,j=1}^n \pp_i (\wa_{ij}(x) \pp_j \wv(x))
+\sum_{j=1}^n \wb_j (x)\pp_j \wv(x)
+ \wc(x)\wv(x),\ \wv \in \Omega. 
\end{equation*}
We assume that $\wa_{ij}\in C^1(\overline{\Omega})$, $\wa_{ij}= \wa_{ji}$ ($1\leq i,j \leq n$) and that there exists a constant $\wmm>0$ such that 
\begin{equation*}
\frac1{\wmm} |\xi|^2 
\leq
\sum_{i,j=1}^n \wa_{ij}(x) \xi_i \xi_j
\leq
\wmm |\xi|^2,
\quad
\xi=(\xi_1,\ldots, \xi_n) \in \R^n,\ 
x \in \overline{\Omega}, 
\end{equation*}
and $\wb_j \in C(\overline{\Omega})$ ($1\leq j\leq n$), $\wc \in C(\overline{\Omega})$. 

Set $\va_0(x):=\va(x,t_0)=e^{\la d(x)}$, $x\in \Omega$. 
Let $D\subset \Omega$ be a sub-domain with the smooth boundary $\pp D$. 
Then we have the following Lemma. 
\begin{lem}
\label{lem:celeme}
There exists $\la_0>0$ such that for any $\la>\la_0$, 
we can choose $s_0(\la)>0$ satisfying: 
there exists 
$C=C(s_0,\la_0)>0$ such that  
\begin{align*}
&
\int_D
\left(
\frac1{s}  
\sum_{i,j=1}^n|\pp_i \pp_j \wv|^2
+s\la^2 |\nabla \wv|^2 
+s^3\la^4 |\wv|^2
\right)
e^{2s \va_0}\,dx 
\\
&
\leq 
C\int_D |\widetilde{L} \wv|^2 e^{2s \va_0}\,dx
+
Ce^{Cs}
\int_{\pp D}
\left(
|\nabla \wv|^2 + |\wv|^2 
\right)
\,dsdt
\end{align*}
for 
all $s> s_0$ and 
all $\wv \in H^2(\Omega)$. 
\end{lem}

\begin{proof}[Proof of Theorem \ref{thm:ce0}]
Since 
$w=\rho_1\pp_t u - L u$ in $Q$, 
we have
\begin{equation*}
\rho_2^2\pp_t u(x,t)- (\rho_1\pp_t-  L )w(x,t) = G(x,t),\quad (x,t)\in Q
\end{equation*}
by \eqref{eq:lem01}. 
So we obtain two parabolic equations with respect to $u$ and $w$:
\begin{align}
\label{eq:ce01}
&
\rho_1\pp_t w(x,t) - L w(x,t) = \rho_2^2\pp_t u(x,t) -G(x,t) ,\quad (x,t)\in Q,\\
\label{eq:ce02}
&
\rho_1\pp_t u(x,t) - L u(x,t) = w(x,t) ,\quad (x,t)\in Q.
\end{align}
Applying the Lemma \ref{lem:celem} 
to \eqref{eq:ce01} and \eqref{eq:ce02}, we have
\begin{align}
\label{eq:ce03}
&\int_V
\left(
\frac{1}{s} |\pp_t w|^2 
+s\la^2|\nabla w|^2
+s^3 \la^4 |w|^2 
\right)e^{2s\va}\,dxdt
\\
&\leq 
C\int_V |\pp_t u|^2 e^{2s\va}\,dxdt
+C\int_V |G|^2 e^{2s\va}\,dxdt 
\nonumber \\
&\quad
+Ce^{Cs}
\int_{\pp V}
\left(
|\nabla w|^2 + |\pp_t w|^2 + |w|^2 
\right)
\,dsdt
\nonumber 
\end{align}
and 
\begin{align}
\label{eq:ce04}
&\int_V
\left[
\frac{1}{s} 
\left(
|\pp_t u|^2 
+\sum_{i,j=1}^n |\pp_i \pp_j u|^2
\right)
+s\la^2|\nabla u|^2
+s^3 \la^4 |u|^2 
\right]e^{2s\va}\,dxdt
\\
&\leq 
C\int_V |w|^2 e^{2s\va}\,dxdt
+Ce^{Cs}
\int_{\pp V}
\left(
 |\nabla u|^2 +|\pp_t u|^2 + |u|^2
\right)
\,dsdt
.
\nonumber 
\end{align}
Substituting the estimate of $|\pp_t u|^2$ in \eqref{eq:ce04} into 
the right-hand side of \eqref{eq:ce03}, we obtain
\begin{align}
&\int_V
\left(
\frac{1}{s} |\pp_t w|^2 
+s\la^2 |\nabla w|^2
+s^3 \la^4 |w|^2 
\right)e^{2s\va}\,dxdt
\nonumber \\
&\leq 
C\int_V s|w|^2 e^{2s\va}\,dxdt
+C\int_V |G|^2 e^{2s\va}\,dxdt 
+Ce^{Cs} B_1 
\nonumber 
\end{align}
where 
\begin{equation*}
B_1=
\int_{\pp V}
\left(
|\nabla w|^2 + |\pp_t w|^2 + |w|^2 
+ |\nabla u|^2 +|\pp_t u|^2 + |u|^2
\right)
\,dsdt
.
\end{equation*}
Taking sufficient large $s>0$, we can absorb the first term on the right-hand side into the left-hand side, 
we have
\begin{align}
\label{eq:ce05}
&\int_V
\left(
\frac{1}{s} |\pp_t w|^2 
+s \la^2|\nabla w|^2
+s^3 \la^4 |w|^2 
\right)
e^{2s\va}\,dxdt
\\
&
\leq 
C\int_V |G|^2 e^{2s\va}\,dxdt
+Ce^{Cs} B_1
\nonumber 
\end{align}
Combining this with \eqref{eq:ce04}, we get
\begin{align}
\label{eq:ce06}
&\int_V
\left[
s^2 \la^4 
\left(
|\pp_t u|^2
+\sum_{i,j=1}^n |\pp_i \pp_j u|^2
\right) 
+s^4 \la^6 |\nabla u|^2
+s^6 \la^8 |u|^2 
\right]e^{2s\va}\,dxdt\\
&
\leq 
C\int_V |G|^2 e^{2s\va}\,dxdt
+Ce^{Cs} B_1.
\nonumber 
\end{align}
By \eqref{eq:ce02} and \eqref{eq:ce05}, we have
\begin{equation}
\label{eq:ce07}
\int_V
\frac1{s} |\pp_t (\rho_1\pp_t u -L u)|^2 
e^{2s\va}\,dxdt 
\leq 
C\int_V |G|^2 e^{2s\va}\,dxdt 
+Ce^{Cs} B_1.
\end{equation}
Setting $u_0=\pp_t u$, the left-hand side of \eqref{eq:ce07} gives us
\begin{equation}
\label{eq:ce10}
\int_V
\frac1{s} |\rho_1\pp_t u_0 -L u_0|^2e^{2s\va}\,dxdt 
\leq 
C\int_V |G|^2 e^{2s\va}\,dxdt 
+Ce^{Cs} B_1.
\end{equation}
By Lemma \ref{lem:celem}, 
we may estimate the left-hand side of \eqref{eq:ce10} from below 
and
we get
\begin{align*}
&
\int_V
\left[
\frac1{s^2} 
\left(
|\pp_t u_0|^2 
+
\sum_{i,j=1}^n |\pp_i\pp_j u_0|^2
\right)
+\la^2 |\nabla u_0|^2
+s^2 \la^4 |u_0|^2 
\right]e^{2s\va}\,dxdt\\
&\leq 
C\int_V |G|^2 e^{2s\va}\,dxdt
+Ce^{Cs}
\int_{\pp V}
\left(
 |\nabla u_0|^2 +|\pp_t u_0|^2 + |u_0|^2
\right)
\,dsdt
+Ce^{Cs} B_1,
\end{align*}
that is,
\begin{align}
\label{eq:ce11}
&\int_V
\left[
\frac1{s^2}  
\left(
|\pp_t^2 u|^2 
+
\sum_{i,j=1}^n |\pp_t\pp_i\pp_j u|^2
\right)
+\la^2|\nabla \pp_t u|^2
+s^2 \la^4 |\pp_t u|^2 
\right]e^{2s\va}\,dxdt
\\
&\leq 
C\int_V |G|^2 e^{2s\va}\,dxdt 
+Ce^{Cs} B.
\nonumber 
\end{align}
Together this with \eqref{eq:ce06}, we have \eqref{eq:pce}. 
\end{proof}

\section{Proof of stability estimate}

We choose a suitable weight functions $\varphi$, that is, $\beta>0$ and a distance function $d$.
Take $\delta>0$ such that 
\begin{equation*}
0<t_0-2\delta<t_0<t_0+2\delta<T.
\end{equation*}
Taking a bounded domain $\Omega_0$ with the smooth boundary $\pp\Omega_0$ such that 
\begin{equation*}
\Omega \subset \Omega_0,\quad 
\Omega \neq \Omega_0,\quad
\overline{\gamma} =\overline{\pp\Omega\cap\Omega_0},\quad
\pp\Omega\setminus\gamma \subset \pp\Omega_0,
\end{equation*} 
we choose an open subset $\omega_0 \Subset \Omega_0\setminus \overline{\Omega}$. 
There exists a distance function $d\in C^2(\R^n)$ such that 
\begin{equation*}
d(x)>0,\  x\in \Omega_0,\quad
d(x)=0,\  x\in \pp\Omega_0,\quad
|\nabla d(x)|>0,\ x\in \overline{\Omega}.
\end{equation*}
The existence of such a function is proved in \cite{FIm} (see the proof of Lemma 1.1 in \cite{FIm}). 
We take $\varepsilon_0\in (0,1)$ such that 
\begin{equation*}
\omega \subset 
\left\{ x\in \Omega_0 \relmiddle| d(x)>\varepsilon \| d \|_{C(\overline{\Omega_0})}
\right\}\cap \overline{\Omega},\quad
\varepsilon \in (0,\varepsilon_0].
\end{equation*}
Take
$\beta>0$ such that 
\begin{equation*}
\frac{\| d \|_{C(\overline{\Omega_0})}}{4\delta^2} <\beta < \frac{\| d \|_{C(\overline{\Omega_0})}}{3\delta^2}.
\end{equation*}
Fixing $\varepsilon \in (0,\varepsilon_0]$, we set 
\begin{equation*}
\mu_k=\varepsilon \left( \frac{k}{3}\| d \|_{C(\overline{\Omega_0})} -\beta\delta^2\right)>0, \quad
k=1,2,3.
\end{equation*}
We set 
$
Q_k=
\{(x,t)\in Q \mid \psi(x,t)>\mu_k \}
$, 
$Q_k^-=
Q_k\cap \{ (x,t)\in Q \mid t <t_0 \}$, 
$
\Omega_k=
Q_k \cap \{ (x,t)\in Q \mid t=t_0\}
$ for $k=1,2,3$. 
Since $0<\mu_1<\mu_2<\mu_3$, we have 
\begin{equation*}
\omega\times (t_0-\sqrt{\varepsilon}\delta,t_0+\sqrt{\varepsilon}\delta) \subset Q_3\subset Q_2\subset Q_1
\subset \overline{\Omega} \times (t_0-2\delta,t_0+2\delta) \subset \overline{Q}
\end{equation*}
and 
$\omega \subset \Omega_3 \subset \Omega_2 \subset \Omega_1 \subset \overline{\Omega}$. 
For this choice of the weight function when we are given the interior domain $\omega$ and the sub-boundary $\gamma$, we may refer to \S 5 in Yamamoto \cite{Y}. 

By Lemma \ref{lem:halftoone}, \eqref{eq:eq01}  gives us the following equation 
\begin{equation}
\label{eq:pr01}
\rho_2^2\pp_t u (x,t)
-(\rho_1\pp_t -L)^2 u(x,t)
=F(x,t),\quad (x,t)\in Q,
\end{equation}
where
\begin{align}
\label{eq:pr02}
F(x,t)
&=\left[ \rho_2\pp_t^\frac12- (\rho_1\pp_t - L) \right] \left(f(x)R(x,t)\right) +\rho_2 f(x)\frac{R(x,0)}{\sqrt{\pi t}}\\
&=
R(x,t) 
\sum_{i,j=1}^n \pp_i(a_{ij}(x)\pp_j f(x)) \nonumber\\
&\quad
+\sum_{j=1}^n \left( 2 \sum_{i=1}^n a_{ij}(x)\pp_i R(x,t) + b_j(x) R(x,t)\right) \pp_j f(x) \nonumber\\
&\quad
+\Biggl[
\rho_2\pp_t^{\frac12}R(x,t)-\rho_1\pp_t R(x,t) 
+\sum_{i,j=1}^n \pp_i(a_{ij}(x)\pp_j R(x,t)) \nonumber \\
&\qquad\quad
\sum_{j=1}^n b_j(x) \pp_j R(x,t)+c(x) R(x,t)
+
\frac{\rho_2 R(x,0)}{\sqrt{\pi t}}
\Biggr]f(x),\  (x,t) \in Q
\nonumber . 
\end{align}
Expanding the left-hand side of \eqref{eq:pr01} , we get the following equation. 
\begin{equation}
\label{eq:pr03}
\rho_2^2\pp_t u (x,t)
-\rho_1^2\pp_t^2 u(x,t) +2\rho_1\pp_t L u(x,t) -L^2 u(x,t)
=F(x,t), 
\quad 
(x,t)\in Q. 
\end{equation}
By the equation \eqref{eq:pr03} at $t=t_0$, we have
\begin{equation}
\label{eq:pr04}
\rho_2^2\pp_t u (x,t_0)
-\rho_1^2\pp_t^2 u(x,t_0) +2\rho_1\pp_t L u(x,t_0) -L^2 u(x,t_0)
=F(x,t_0),\quad x\in \Omega.
\end{equation} 
Taking the weighted $L^2$ norm of \eqref{eq:pr04} in $\Omega_2$, 
we obtain
\begin{align}
\label{eq:pr05}
&\int_{\Omega_2} |F(x,t_0)|^2e^{2s\va(x,t_0)}\,dx \\
&
\leq 
C 
\int_{\Omega_2} |\pp_t u(x,t_0)|^2e^{2s\va(x,t_0)}\,dx
+
C
\int_{\Omega_2} |\pp_t^2 u(x,t_0)|^2e^{2s\va(x,t_0)}\,dx
\nonumber \\
&+
C 
\int_{\Omega_2} |\pp_t L u(x,t_0)|^2e^{2s\va(x,t_0)}\,dx
+
C
\int_{\Omega_2} \sum_{|\alpha|\leq 4} |\pp_x^\alpha u(x,t_0)|^2e^{2s\va(x,t_0)}\,dx .
\nonumber 
\end{align}
Henceforth we estimate from the first term to the third term on the right-hand side of \eqref{eq:pr05} 
by the Carleman estimate (Theorem \ref{thm:ce0}). 

To use the Carleman estimate, 
we introduce a cut-off function 
$\chi\in C^\infty(\R^{n+1})$ such that  
$0 \leq \chi \leq 1$ in $\R^{n+1}$,  
$\supp \chi \subset \{(x,t)\in R^{n+1}\mid \psi(x,t)>\mu_1\}$ and 
$\chi \equiv 1$ in $\{(x,t)\in R^{n+1}\mid \psi(x,t)>\mu_2\}$. 

Set $y=\chi\pp_t u$, $z=\chi \pp_t^2 u$. 
By \eqref{eq:pr01}, we have
\begin{align}
\label{eq:pr06}
&\rho_2^2\pp_t y (x,t)
-\rho_1^2\pp_t^2 y(x,t) +2\rho_1\pp_t L y(x,t) -L^2 y(x,t) \\
&=\chi \pp_t F(x,t) + h_1(x,t), \quad (x,t)\in Q, \nonumber\\
\label{eq:pr07}
&\rho_2^2\pp_t z (x,t)
-\rho_1^2\pp_t^2 z(x,t) +2\rho_1\pp_t L z(x,t) -L^2 z(x,t)  \\
&=\chi \pp_t^2 F(x,t) + h_2(x,t), \quad (x,t)\in Q, \nonumber
\end{align}
where $h_1$ and $h_2$ are linear combinations of 
$\pp_x^\alpha \pp_t u$ ($|\alpha|\leq 3$), $\pp_x^\alpha \pp_t^2 u$ ($|\alpha|\leq 1$)
and 
$\pp_x^\alpha \pp_t^2 u$ ($|\alpha|\leq 3$), $\pp_x^\alpha \pp_t^3 u$ ($|\alpha|\leq 1$),
respectively, 
with coefficients containing 
$\pp_x^\alpha \chi$ ($1\leq |\alpha| \leq 4$), 
$\pp_x^\alpha \pp_t \chi$ ($|\alpha| \leq 2$). 

Fixing $\la>0$ and applying Theorem \ref{thm:ce0} to \eqref{eq:pr06} and \eqref{eq:pr07} in $Q_1$, we have
\begin{align}
\label{eq:pr08}
&
\int_{Q_1}
\Biggl[
s^2 \left( |\pp_t y|^2 +|\pp_t z|^2 \right)  
+s^2 \left( \sum_{i,j=1}^n|\pp_i \pp_j y|^2
+\sum_{i,j=1}^n|\pp_i \pp_j z|^2
\right) 
\\
&\qquad \quad
+s^4 \left(|\nabla y|^2 +|\nabla z|^2\right) 
+s^6  \left(|y|^2 +|z|^2 \right)
\Biggr]
e^{2s\va}\,dxdt 
\nonumber\\
&
\leq 
C\int_{Q_1} \chi^2 \left( | \pp_t F |^2 + | \pp_t^2 F |^2 \right) e^{2s\va}\,dxdt 
\nonumber\\
&\quad
+
C\int_{Q_1}  \left( | h_1 |^2 + | h_2 |^2 \right) e^{2s\va}\,dxdt. 
\nonumber
\end{align}
Here we note that the boundary term on $\pp Q_1$ of the Carleman estimate vanishes. 
Indeed, by the choice of $\chi$, the boundary term on $\pp Q_1\setminus \Gamma$ becomes $0$. 
Since $f=\pp_\nu f=0$ on $\gamma$, \eqref{eq:eq03} and $\pp_x^{\alpha}u=0$ on $\Gamma$, $|\alpha|= 2$, 
moreover, 
the boundary term on $\pp Q_1 \cap \Gamma$ vanishes. 
 
Noting that 
$
h_1(x,t)=h_2 (x,t)=0$, $(x,t)\in Q_2
$, 
\begin{equation}
\label{eq:ineq_weight}
1 \leq e^{2s\va(x,t)} \leq e^{2s \exp(\lambda \mu_2)},
\quad (x,t)\in Q_1 \setminus Q_2, 
\end{equation}
and that $h_1,h_2 \in L^2 (Q_1)$ have an upper bound depending on $M$ by \eqref{eq:M},
we obtain 
\begin{align}
\label{eq:pr09}
\int_{Q_1}  \left( | h_1 |^2 + | h_2 |^2 \right) e^{2s\va}\,dxdt 
&
=
\int_{Q_1\setminus Q_2}  \left( | h_1 |^2 + | h_2 |^2 \right) e^{2s\va}\,dxdt
\\
&\leq 
Ce^{2s \exp(\lambda \mu_2)}. 
\nonumber
\end{align}
By \eqref{eq:R} and \eqref{eq:pr02}, we get
\begin{align}
\label{eq:pr10}
&
\int_{Q_1} \chi^2 \left( | \pp_t F |^2 + | \pp_t^2 F |^2 \right) e^{2s\va}\,dxdt 
\\
&
\leq 
C\int_{Q_2}  \sum_{|\alpha|\leq 2} | \pp_x^\alpha f |^2 e^{2s\va}\,dxdt
+
C\int_{Q_1\setminus Q_2}  \sum_{|\alpha|\leq 2} | \pp_x^\alpha f |^2 e^{2s\va}\,dxdt \nonumber \\
&
\leq 
C\int_{Q_2}  \sum_{|\alpha|\leq 2} | \pp_x^\alpha f |^2 e^{2s\va}\,dxdt
+
Ce^{2s \exp(\lambda \mu_2)}. 
\nonumber
\end{align}
In the last in equality of \eqref{eq:pr10}, we used the following inequality
\begin{equation*}
\int_{Q_1\setminus Q_2}  \sum_{|\alpha|\leq 2} | \pp_x^\alpha f |^2 e^{2s\va}\,dxdt
\leq Ce^{2s \exp(\lambda \mu_2)}. 
\end{equation*}
which is obtained by \eqref{eq:M} and \eqref{eq:ineq_weight}. 
Combining \eqref{eq:pr08} with \eqref{eq:pr09} and \eqref{eq:pr10}, we have
\begin{align}
&
\int_{Q_1}
\Biggl[
s^2\left( |\pp_t y|^2 +|\pp_t z|^2 \right)  
+s^2 \left( \sum_{i,j=1}^n|\pp_i \pp_j y|^2
+\sum_{i,j=1}^n|\pp_i \pp_j z|^2
\right) \nonumber   \\
&\qquad \quad
+s^4  \left(|\nabla y|^2 +|\nabla z|^2\right) 
+s^6  \left(|y|^2 +|z|^2 \right)
\Biggr]
e^{2s\va}\,dxdt 
\nonumber \\
&
\leq 
C\int_{Q_2}  \sum_{|\alpha|\leq 2} | \pp_x^\alpha f |^2 e^{2s\va}\,dxdt
+
Ce^{2s \exp(\lambda \mu_2)}. 
\nonumber 
\end{align}
Since $y=\pp_t u$, $z=\pp_t^2 u$ in $Q_2$ and $Q_2 \subset Q_1$, we obtain
\begin{align}
\label{eq:pr11}
&
\int_{Q_2}
\Biggl[
s^2  \left( |\pp_t^2 u|^2 +|\pp_t^3 u|^2 \right)  
+s^2 \left( \sum_{i,j=1}^n|\pp_i \pp_j \pp_t u|^2
+\sum_{i,j=1}^n|\pp_i \pp_j \pp_t^2 u|^2
\right)  
\\
&\quad \quad
+s^4  \left(|\nabla \pp_t u|^2 +|\nabla \pp_t^2 u|^2\right) 
+s^6  \left(|\pp_t u|^2 +|\pp_t^2 u|^2 \right)
\Biggr]
e^{2s\va}\,dxdt 
\nonumber\\
&
\leq 
C\int_{Q_2}  \sum_{|\alpha|\leq 2} | \pp_x^\alpha f |^2 e^{2s\va}\,dxdt
+
Ce^{2s \exp(\lambda \mu_2)}. 
\nonumber
\end{align}
Since 
$\chi =0$ in $\Omega \times (0,t_0) \setminus Q_1^-$,
we have
\begin{align}
\label{eq:pr12}
&\int_{\Omega_2}
|\pp_t u(x,t_0)|^2 e^{2s\va(x,t_0)}
\,dx\\
&\leq 
\int_{\Omega_1}
|\chi(x,t_0) \pp_t u(x,t_0)|^2 e^{2s\va(x,t_0)}
\,dx \nonumber \\
&=
\int_0^{t_0}
\int_{\Omega_1} 
\pp_t \left( |\chi \pp_t u|^2 e^{2s\va}\right)
\,dxdt \nonumber \\
&\leq
C
\int_0^{t_0}
\int_{\Omega_1} 
\left[ \chi^2|\pp_t u||\pp_t^2 u| + \left(\chi |\pp_t \chi| +s \chi^2\right)|\pp_t u|^2 \right]e^{2s\va}
\,dxdt \nonumber \\
&=
C
\int_{Q_1^-} 
\left[ \chi^2|\pp_t u||\pp_t^2 u| + \left(\chi |\pp_t \chi| +s \chi^2\right)|\pp_t u|^2 \right]e^{2s\va}
\,dxdt \nonumber \\
&\leq
C
\int_{Q_1} 
s \left( |\pp_t u|^2 + |\pp_t^2 u|^2\right) e^{2s\va}
\,dxdt \nonumber \\
&\leq
C
\int_{Q_2} 
s \left( |\pp_t u|^2 + |\pp_t^2 u|^2\right) e^{2s\va}
\,dxdt 
+C e^{2s\exp(\la \mu_2)}
\nonumber .
\end{align}
Together this with \eqref{eq:pr11}, we see that 
\begin{equation}
\label{eq:pr13}
\int_{\Omega_2}
|\pp_t u(x,t_0)|^2 e^{2s\va(x,t_0)}
\,dx
\leq 
\frac{C}{s^5}\int_{Q_2}  \sum_{|\alpha|\leq 2} | \pp_x^\alpha f |^2 e^{2s\va}\,dxdt
+
Ce^{2s \exp(\lambda \mu_2)}. 
\end{equation}
Similarly, we may obtain
\begin{align}
\label{eq:pr14}
&\int_{\Omega_2}
|\pp_t^2 u(x,t_0)|^2 e^{2s\va(x,t_0)}
\,dx\\
&\leq 
\frac{C}{s}\int_{Q_2}  \sum_{|\alpha|\leq 2} | \pp_x^\alpha f |^2 e^{2s\va}\,dxdt
+
Ce^{2s \exp(\lambda \mu_2)},
\nonumber \\
\label{eq:pr15}
&\int_{\Omega_2}
|\pp_t L u(x,t_0)|^2 e^{2s\va(x,t_0)}
\,dx\\
&\leq 
\frac{C}{s}\int_{Q_2}  \sum_{|\alpha|\leq 2} | \pp_x^\alpha f |^2 e^{2s\va}\,dxdt
+
Ce^{2s \exp(\lambda \mu_2)}.
\nonumber
\end{align}
By \eqref{eq:pr05}, \eqref{eq:pr13}--\eqref{eq:pr15}, we have
\begin{align}
\label{eq:pr16}
&\int_{\Omega_2} |F(x,t_0)|^2e^{2s\va(x,t_0)}\,dx \\
&
\leq 
\frac{C}{s}\int_{Q_2}  \sum_{|\alpha|\leq 2} | \pp_x^\alpha f |^2 e^{2s\va}\,dxdt
\nonumber \\
&\quad
+
C
\int_{\Omega_2} \sum_{|\alpha|\leq 4} |\pp_x^\alpha u(x,t_0)|^2e^{2s\va(x,t_0)}\,dx 
+
Ce^{2s \exp(\lambda \mu_2)}.
\nonumber 
\end{align}

Next we estimate 
\begin{equation*}
\int_{\Omega_2} |F(x,t_0)|^2e^{2s\va(x,t_0)}\,dx 
\end{equation*}
from below. 
To apply the Carleman estimate for elliptic equation, 
we introduce a cut-off function 
$\wchi\in C^\infty(\R^n)$ such that  
$0 \leq \wchi \leq 1$ in $\R^n$, 
$\supp \wchi \subset \{x\in R^{n}\mid d(x)>\mu_1\}$ and 
$\wchi\equiv 1$ in $\{x\in R^{n}\mid d(x)>\mu_2\}$. 

Set $\wf=\wchi f$. By \eqref{eq:pr02} at $t=t_0$, we have
\begin{align}
\label{eq:pr17}
&
\sum_{i,j=1}^n \pp_i(a_{ij}(x)\pp_j \wf(x)) \\
&\quad
+\frac{1}{R(x,t_0)}\sum_{j=1}^n \left( 2 \sum_{i=1}^n a_{ij}(x)\pp_i R(x,t_0) + b_j(x) R(x,t_0) \right) \pp_j \wf(x) \nonumber\\
&\quad
+\frac{1}{R(x,t_0)}\Biggl[
\rho_2\pp_t^{\frac12}R(x,t_0)-\rho_1\pp_t R(x,t_0) 
+\sum_{i,j=1}^n \pp_i(a_{ij}(x)\pp_j R(x,t_0)) \nonumber \\
&\qquad\qquad\qquad
+\sum_{j=1}^n b_j(x) \pp_j R(x,t_0)+c(x) R(x,t_0)
+
\frac{\rho_2 R(x,0)}{\sqrt{\pi t_0}}
\Biggr]\wf(x) \nonumber \\
&=
\frac{\wchi (x) F(x,t_0)}{R(x,t_0)} + h_3(x), \quad x\in \Omega, 
\nonumber
\end{align}
where 
\begin{align*}
h_3(x)
&=
2\sum_{i,j=1}^n a_{ij} \pp_i \wchi(x)\pp_j f(x)\\
&\quad
+\Biggl[
\sum_{i,j=1}^n \pp_i (a_{ij}\pp_j \wchi(x)) \\
&\qquad \quad
+\frac{1}{R(x,t_0)}\sum_{j=1}^n \left( 2 \sum_{i=1}^n a_{ij}(x)\pp_i R(x,t_0) + b_j(x) R(x,t_0) \right) (\pp_j \wchi)
\Biggr]f(x),\quad x\in \Omega.
\end{align*}
Here we see that $h_3(x)=0$, $x\in \Omega_2$ and $\|h_3\|_{L^2(\Omega)}\leq C$ by \eqref{eq:M}. 
Applying the Lemma \ref{lem:celeme} to \eqref{eq:pr17} in $\Omega_1$ 
and noting that $\wf(x)=f(x)$, $x\in \Omega_2$,  \eqref{eq:R}, \eqref{eq:M} and \eqref{eq:ineq_weight}, we obtain 
\begin{align}
\label{eq:pr18}
&
\frac1{s}
\int_{\Omega_2}
 \sum_{|\alpha|\leq 2} | \pp_x^\alpha f(x) |^2 e^{2s\va(x,t_0)}
\,dx \\
&\leq
C
\int_{\Omega_2}
\left(
\frac1{s}  
\sum_{i,j=1}^n|\pp_i \pp_j f|^2
+s |\nabla f|^2 
+s^3 |f|^2
\right)
e^{2s\va(x,t_0)}\,dx  \nonumber \\
&\leq
C
\int_{\Omega_1}
\left(
\frac1{s}  
\sum_{i,j=1}^n|\pp_i \pp_j \wf|^2
+s |\nabla \wf|^2 
+s^3 |\wf|^2
\right)
e^{2s\va(x,t_0)}\,dx 
\nonumber \\
&\leq 
C
\int_{\Omega_1}
\left|
\frac{\wchi (x) F(x,t_0)}{R(x,t_0)}
\right|^2
e^{2s\va(x,t_0)}\,dx 
+C
\int_{\Omega_1\setminus \Omega_2}
\left|
h_3(x)
\right|^2
e^{2s\va(x,t_0)}\,dx 
\nonumber \\
&\leq
C
\int_{\Omega_2}
\left|
F(x,t_0)
\right|^2
e^{2s\va(x,t_0)}\,dx 
+
C
\int_{\Omega_1\setminus \Omega_2}
\left|
F(x,t_0)
\right|^2
e^{2s\va(x,t_0)}\,dx \nonumber\\
&\quad
+Ce^{2s \exp(\lambda \mu_2)} \nonumber \\
&\leq
C
\int_{\Omega_2}
\left|
F(x,t_0)
\right|^2
e^{2s\va(x,t_0)}\,dx
+Ce^{2s \exp(\lambda \mu_2)}. \nonumber  
\end{align}
Here the boundary term on $\pp \Omega_1$ of the Carleman estimate vanishes. 
Indeed, by $f=\pp_\nu f=0$ on $\gamma$ and the choice of $\wchi$, we see that $\wf=\pp_\nu \wf=0$ on $\pp \Omega_1$.

By \eqref{eq:pr16} and \eqref{eq:pr18}, we obtain
\begin{align}
\label{eq:pr19}
&\frac1{s}\int_{\Omega_2}  \sum_{|\alpha|\leq 2} | \pp_x^\alpha f(x) |^2 e^{2s\va(x,t_0)}\,dx \\
&
\leq 
\frac{C}{s}\int_{Q_2}  \sum_{|\alpha|\leq 2} | \pp_x^\alpha f |^2 e^{2s\va}\,dxdt
\nonumber \\
&\quad
+
C
\int_{\Omega_2} \sum_{|\alpha|\leq 4} |\pp_x^\alpha u(x,t_0)|^2e^{2s\va(x,t_0)}\,dx 
+
Ce^{2s \exp(\lambda \mu_2)}.
\nonumber 
\end{align}
Let us estimate the first integral term on the right-hand side of \eqref{eq:pr19}. 
\begin{align*}
\int_{Q_2}  \sum_{|\alpha|\leq 2} | \pp_x^\alpha f |^2 e^{2s\va}\,dxdt
\leq 
\int_{\Omega_2}  \sum_{|\alpha|\leq 2} | \pp_x^\alpha f |^2 e^{2s\va}e^{2s\va(x,t_0)} \left( \int_0^T e^{-2s(\va(x,t_0)-\va(x,t))}\,dt \right)\,dx
\end{align*}
Noting that $\va(x,t_0)-\va(x,t)=e^{\la d(x)} (1-e^{-\beta(t-t_0)^2}) \geq 0$, $(x,t)\in Q$, 
by Lebesgue's dominated convergence theorem, 
$
\int_0^T e^{-2s(\va(x,t_0)-\va(x,t))}\,dt 
$
converges pointwise to $0$ in $\Omega$ as $s$ tends to $\infty$. 
By Dini's theorem, moreover, we may see that 
$
\int_0^T e^{-2s(\va(x,t_0)-\va(x,t))}\,dt 
$
converges uniformly to $0$ in $\Omega_2$ as $s$ goes to $\infty$.
Hence 
taking sufficient large $s>0$, we may absorb the first term on the right-hand side of \eqref{eq:pr19} 
into the left-hand side, that is, there exists $s_1>0$ such that 
\begin{align}
\label{eq:pr20}
&\int_{\Omega_2}  \sum_{|\alpha|\leq 2} | \pp_x^\alpha f(x) |^2 e^{2s\va(x,t_0)}\,dx\\
&
\leq 
C
\int_{\Omega_2} \sum_{|\alpha|\leq 4} |\pp_x^\alpha u(x,t_0)|^2e^{2s\va(x,t_0)}\,dx 
+
Ce^{2s \exp(\lambda \mu_2)}
\nonumber 
\end{align}
for all $s>s_1$. 

By $e^{2s\va(x,t_0)}\geq e^{2s\exp(\la\mu_3)}$, $x\in \Omega_3$, we estimate the left-hand side of \eqref{eq:pr20} from below. 
\begin{align*}
\int_{\Omega_2}  \sum_{|\alpha|\leq 2} | \pp_x^\alpha f(x) |^2 e^{2s\va(x,t_0)}\,dx
&\geq 
\int_{\Omega_3}  \sum_{|\alpha|\leq 2} | \pp_x^\alpha f(x) |^2 e^{2s\va(x,t_0)}\,dx\\
&\geq 
e^{2s\exp(\la\mu_3)}  \int_{\Omega_3}\sum_{|\alpha|\leq 2} | \pp_x^\alpha f(x) |^2 \,dx. 
\end{align*}
Combining this with \eqref{eq:pr20}, we have
\begin{align}
\label{eq:pr21}
&e^{2s\exp(\la\mu_3)}  \int_{\Omega_3}\sum_{|\alpha|\leq 2} | \pp_x^\alpha f(x) |^2 \,dx\\
&
\leq 
C
\int_{\Omega_2} \sum_{|\alpha|\leq 4} |\pp_x^\alpha u(x,t_0)|^2e^{2s\va(x,t_0)}\,dx 
+
Ce^{2s \exp(\lambda \mu_2)}
\nonumber 
\end{align}
for all $s>s_1$. 
Dividing the both side of \eqref{eq:pr21} by $e^{2s\exp(\la\mu_3)} $, we get
\begin{align}
\label{eq:pr22}
&\int_{\Omega_3}\sum_{|\alpha|\leq 2} | \pp_x^\alpha f(x) |^2 \,dx\\
&
\leq 
C
\int_{\Omega_2} \sum_{|\alpha|\leq 4} |\pp_x^\alpha u(x,t_0)|^2 e^{2s[\va (x,t_0)-\exp(\lambda \mu_3)]}\,dx 
+
Ce^{2s [\exp(\lambda \mu_2)-\exp(\lambda \mu_3)]}
\nonumber 
\end{align}
for all $s>s_1$. 
Noting that 
$\va (x,t_0)-\exp(\lambda \mu_3)>0$, $x\in \Omega_2$ and 
$\exp(\lambda \mu_2)-\exp(\lambda \mu_3)<0$, 
there exist $C_1,C_2,D_1,D_2>0$ such that 
\begin{equation*}
\left\|
f
\right\|_{H^2(\Omega_3)}^2
\leq 
C_1
\left\|
u(\cdot,t_0)
\right\|_{H^4(\Omega_2)}^2 e^{D_1 s}
+
C_2
e^{-D_2 s}
\end{equation*}
for all $s>s_1$. 
Taking $C_3>0$ such that $C_1\leq C_3e^{-D_1s_1}$, $C_2 \leq C_3e^{D_2 s_1}$ 
and setting $\sigma =s-s_1$, we have 
\begin{equation}
\label{eq:pr23}
\left\|
f
\right\|_{H^2(\Omega_3)}^2
\leq 
C_3\left(
\left\|
u(\cdot,t_0)
\right\|_{H^4(\Omega_2)}^2 e^{D_1 \sigma}
+
e^{-D_2 \sigma}
\right)
\end{equation}
for all $\sigma>0$. 

If 
$
\left\| u(\cdot,t_0) \right\|_{H^4(\Omega_2)} 
\geq 1
$, by 
$\| f \|_{H^2(\Omega)}
\leq M$, we may obtain the stability estimate immediately. Moreover, provided that 
$\left\|
u(\cdot,t_0)
\right\|_{H^4(\Omega_2)} =0
$, we have 
\begin{equation}
\label{eq:pr24}
\left\|
f
\right\|_{H^2(\Omega_3)}^2
\leq 
C_3
e^{-D_2 \sigma}
\end{equation}
by \eqref{eq:pr23}. As $\sigma$ goes to $\infty$, the right hand side of \eqref{eq:pr24} tends to $0$. 
Hence we get $\left\|
f
\right\|_{H^2(\Omega_3)}=0$. 
So, it is sufficient to assume that 
$0<
\left\|
u(\cdot,t_0)
\right\|_{H^4(\Omega_2)} 
< 1
$. 

Taking 
\begin{equation*}
\sigma = -\frac{\log \left( \left\| u(\cdot,t_0) \right\|_{H^4(\Omega_2)}^2  \right)}{D_1+D_2}>0,
\end{equation*}
which minimize the right-hand side of \eqref{eq:pr23}, we see that 
\begin{equation*}
\left\|
u(\cdot,t_0)
\right\|_{H^4(\Omega_2)}^2 e^{D_1 \sigma}
+
e^{-D_2 \sigma}
=2 \left\|
u(\cdot,t_0)
\right\|_{H^4(\Omega_2)}^{2\kappa}
\end{equation*}
with $\kappa = \frac{D_2}{D_1+D_2} \in (0,1)$. 
Together this with \eqref{eq:pr23}, we have
\begin{equation*}
\left\|
f
\right\|_{H^2(\Omega_3)}
\leq 
C
\left\|
u(\cdot,t_0)
\right\|_{H^4(\Omega_2)}^{\kappa}. 
\end{equation*}
Since $\omega \subset \Omega_3$ and $\Omega_2 \subset \Omega$, 
we get \eqref{eq:se}. 
Thus we complete the proof. \qed

%
%
%
%

\end{document}